\documentclass[12pt]{article}
\usepackage{latexsym,amsfonts,amsmath,amsthm,amssymb,epsfig}
\usepackage{color}

\usepackage{marvosym}

\usepackage[normalem]{ulem}

\newtheorem{thm}{Theorem}
\newtheorem*{thm*}{Theorem}

\newtheorem{prop}[thm]{Proposition}
\newtheorem{cor}[thm]{Corollary}

\newtheorem{OP}{Open Problem}

\theoremstyle{remark}
\newtheorem{rem}{Remark}

\theoremstyle{definition}

\newcommand{\X}{\mathcal{X}} 
\newcommand{\R}{\mathbb{R}}

\newcommand{\N}{\mathbb{N}}

\newcommand{\Z}{\mathbb{Z}}
\newcommand{\E}{\mathbb{E}}

\newcommand{\diag}{{\rm diag\ }}

\newcommand{\esssup}{\operatornamewithlimits{ess\,sup}}

\newcommand{\e}{{\varepsilon}}

\title{Lower Bounds for the Error of 
Quadrature Formulas for Hilbert Spaces} 

\author{Aicke Hinrichs\footnote{Institut f\"ur Analysis, 
Johannes Kepler Universit\"at Linz, 
Altenbergerstrasse~69, 4040~Linz, Austria, 
Email: \texttt{aicke.hinrichs@jku.at}, \texttt{david.krieg@jku.at}.}
,
David Krieg$^*$,
Erich Novak\footnote{Mathematisches Institut, FSU Jena, 
Ernst-Abbe-Platz 2, 07740 Jena, Germany, 
Email: \texttt{erich.novak@uni-jena.de}.}
, 
Jan Vyb\'iral\footnote{Dept.\ of Mathematics FNSPE, Czech Technical University 
in Prague, Trojanova 13, 12000 Prague, Czech Republic, 
Email: \texttt{jan.vybiral@fjfi.cvut.cz}.
}}

\begin{document}

\date{\today} 

\maketitle

\begin{abstract}
We prove lower bounds for the worst case error of quadrature formulas 
that use given sample points 
$\X_n = \{ x_1, \dots , x_n \}$. 
We are mainly interested in optimal point sets $\X_n$,
but also prove lower bounds that hold
with high probability for sets of independently
and uniformly distributed points.
As a tool,
we use a recent result 
(and extensions thereof) of 
Vyb\'iral on the positive semi-definiteness of certain matrices related
to the product theorem of Schur.
The new technique also works for spaces of analytic functions 
where known methods based on decomposable kernels cannot be applied.
\end{abstract}


\thispagestyle{empty}
\newpage

\section{Introduction}\label{Sec:1}

We study error bounds for quadrature formulas and assume 
that the integrand is from a Hilbert space $F$ of 
real valued functions defined
on a set~$D$. 
We assume that function evaluation is continuous and hence are dealing 
with a reproducing kernel Hilbert space (RKHS) $F$ with a kernel $K$. 
We want to compute $S(f)$ for $f \in F$, where $S$ is a 
continuous
linear functional, 
hence 
$S (f)  =\langle f , h \rangle$ 
for some $h\in F$. 
We consider,  
for $c \in \R^n$ and 
$\X_n = \{ x_1, \dots , x_n \} \subset D$,
quadrature formulas $Q_{c, \X_n}: F\to \R$
defined by 
\[
 Q_{c, \X_n}(f)= \sum_{j=1}^n c_j f(x_j).
\]
Then the worst case error (on the unit ball of $F$) of $Q_{c, \X_n}$
is defined by  
\begin{align*}
e(Q_{c, \X_n}, S ) &:=\sup_{||f||\le 1}|S(f)-Q_{c, \X_n}(f)|.
\end{align*} 
If we fix a set $\X_n \subset D$ of sample points we may define 
the radius of information 
$e(\X_n, S)$ by
\[
e(\X_n, S) = \inf_{c\in\R^n} e(Q_{c, \X_n}, S  ) .
\]
Our main interest is in the optimization of $\X_n$ as well as 
of the weights $c$. 
Then we obtain the $n$th minimal worst case error
\[
e(n, S) = \inf_{\X_n\subset D} e(\X_n,S) = \inf_{c\in\R^n} \inf_{\X_n\subset D} e(Q_{c, \X_n},S) .
\]
The minimal number of sample points that are needed 
to achieve a fixed error demand $\e \ge 0$ is described by the number
\[
 n(\e, S) = \min \left\{ n \in \N \mid e(n,S) \le \e  \right\}.
\]

We are mainly interested in 
\emph{tensor product problems}. We will therefore assume that 
$F_i$ is a RKHS
on a domain $D_i$ with kernel $K_i$ for all $i\le d$
and that $\mathbf F_d$ is the tensor product of these spaces.
That is, $\mathbf F_d$ is a RKHS on $\mathbf D_d=D_1\times\cdots\times D_d$
with reproducing kernel
\[
 \mathbf K_d: \mathbf D_d\times \mathbf D_d \to \R,\quad \mathbf K_d(x,y)=\prod_{i=1}^d K_i(x^i,y^i).
\] 
If $h_i\in F_i$ and $S_i(f)=\langle f,h_i\rangle$ 
for $f\in F_i$, 
we will denote by $\mathbf h_d$
the tensor product of the functions $h_i$, i.e.,
$$
\mathbf h_d(t)=(h_1\otimes \dots\otimes h_d)(t)=h_1(t^1)\cdot\ldots\cdot h_d(t^d),
\quad t=(t^1,\dots,t^d)\in  \mathbf D_d. %
$$
We study the tensor product functional $\mathbf S_d=\langle\cdot,\mathbf h_d\rangle$ on $\mathbf F_d$.
Note that
in this paper we assume that $\mathbf S_d$ is a tensor product \emph{functional},
but the results can also be applied to \emph{operators}, see \cite{NW16}.

The complexity of the tensor product problem
is given by the numbers $e(n,\mathbf S_d)$ and $n(\e, \mathbf S_d)$ 
and has been studied in many papers 
for a long time. Traditionally, the functional $\mathbf S_d$ and the dimension~$d$ was fixed 
and the interest was on large~$n$. Here we are mainly interested in 
the \emph{curse of dimensionality}: 
Do we need exponentially many (in $d$) function values 
to obtain an error $\e$ when we fix the error demand and vary the dimension? 

To answer this question one has to prove the corresponding upper or lower bounds. 
\emph{Upper bounds} for specific problems can often be proved by quasi Monte Carlo methods, 
see~\cite{DKS}. In addition there exists a general method, the analysis of the Smolyak algorithm, 
see~\cite{NW10,WW95} and the recent supplement~\cite{NW18}. 

In this paper we concentrate on \emph{lower bounds}, again for a fixed error demand $\e$ 
and (possibly) large dimension. 
Such bounds were first studied in \cite{NSW97} for certain special problems and 
later in \cite{NW01} with the technique of \emph{decomposable kernels}. 
This technique is rather general as long as we consider finite smoothness.
The technique does not work, however, for analytic functions.

In contrast, the approach of \cite{Vy19} %
also works for polynomials and other analytic functions. 
We continue this approach since   %
it opens the door for more lower bounds under general assumptions. 
One result of this paper (Theorem~\ref{thm:abstract-trigonometric-weighted}) reads as follows: 

\theoremstyle{theorem}
\newtheorem*{thm:intro}{Theorem~\ref{thm:abstract-trigonometric-weighted}}
\begin{thm:intro}
For all $i\le d$, let $F_i$ be a RKHS and let $S_i$ be a bounded linear functional on $F_i$ 
with unit norm and nonnegative representer~$h_i$. 
Assume that there are functions $f_i$ and $g_i$ in $F_i$ and a number $\alpha_i\in (0,1]$ such that
$(h_i,f_i,g_i)$ is orthonormal in $F_i$ and $\alpha_i h_i=\sqrt{f_i^2+g_i^2}$.
Then the tensor product problem $\mathbf S_d=S_1\otimes\hdots\otimes S_d$ %
satisfies for all $n\in\N$ that
\[
	e(n,\mathbf S_d)^2 \ge 1-n\prod_{i=1}^d(1+\alpha_i^2)^{-1}.
\]
\end{thm:intro}

In particular, if all the $\alpha_i$'s are equal to some $\alpha>0$ and we want to achieve $e(n,{\mathbf S}_d)\le \varepsilon$
for some $0<\varepsilon<1$,
we obtain
\[
n(\e, \mathbf S_d)\ge (1-\varepsilon^2)(1+\alpha^2)^d.
\]
This implies the curse of dimensionality.
As an application, we use this result to obtain lower bounds for the complexity of the integration
problem on Korobov spaces with increasing smoothness, see Section~\ref{sec:increasing-smoothness}.
These lower bounds complement existing upper bounds from~\cite[Section 10.7.4]{NW10}.

We add in passing that lower bounds of this form are known and much easier 
to prove for quadrature formulas that only have positive weights, 
see Theorem 10.2 of \cite{NW10}.
  
\medskip

The paper is organized as follows. We first
provide a general connection between the worst case error of quadrature formulas
and the positive semi-definiteness of certain matrices in Section~\ref{sec:positive-definiteness}.
We then turn to tensor product problems.
We start with homogeneous tensor products (i.e., all factors $F_i$ and $h_i$ are equal),
see Section~\ref{sec:homo}, where we also consider several examples.
The non-homogeneous case is then discussed in Section~\ref{sec:non-homo}. 
This section also contains the results for Korobov spaces with increasing smoothness.
Section~\ref{sec:homo} and Section~\ref{sec:non-homo} are based on a recent generalization of
Schur's product theorem from~\cite{Vy19}.
In Section~\ref{sec:new-schur}, we discuss further generalizations of Schur's theorem
and possible applications to numerical integration.
Finally, in Section~\ref{sec:random}, we consider lower bounds for the error of quadrature formulas 
that use random point sets (as opposed to optimal point sets).
This allows us to approach situations where we conjecture but cannot prove
the curse of dimensionality for optimal point sets.

\section{Lower bounds and positive definiteness}
\label{sec:positive-definiteness}

We begin with a somewhat surprising result:
Lower bounds for the worst case error of quadrature formulas are equivalent 
to the statement that certain matrices are positive semi-definite. 

\begin{prop}\label{prop:error_vs_definiteness}
Let $F$ be a RKHS on $D$ with kernel $K$ and let $S =\langle \cdot, h \rangle$ 
for some $h\in F$ and $\alpha>0$.
\begin{enumerate}
\item[(i)] The following statements are equivalent for all %
$\X_n=\{x_1,\dots,x_n\} \subset D$.
\begin{itemize}
\item The matrix $\displaystyle \big( K(x_j,x_k) 
- \alpha h(x_j)h(x_k)\big)_{j,k\leq n}$ is positive semi-definite,
\item $\displaystyle e( \X_n ,S)^2 \geq \Vert h \Vert^2 - \alpha^{-1}$.
\end{itemize}
\item[(ii)] The following statements are equivalent for all %
$n\in \N$.
\begin{itemize}
\item The matrix $\displaystyle \big( K(x_j,x_k) 
- \alpha h(x_j)h(x_k)\big)_{j,k\leq n}$ %
is positive 
semi-definite for all $x_1,\hdots,x_n\in D$,
\item $\displaystyle e(n,S)^2 \geq \Vert h \Vert^2 - \alpha^{-1}$.
\end{itemize}
\end{enumerate}
\end{prop}

\begin{proof} To prove the first part, we fix $\X_n=\{x_1,\dots,x_n\}\subset D$.
For $c \in \R^n$ consider the quadrature rule $Q_{c,\X_n}: F\to \R$ with
\[
 Q_{c,\X_n}(f)= \sum_{j=1}^n c_j f(x_j).
\]
Clearly, we have
\begin{align*}
e(Q_{c,\X_n},S)^2&=\sup_{||f||\le 1} \big|S(f)-Q_{c,\X_n}(f)\big|^2
=\biggl\|h-\sum_{j=1}^n c_j K(x_j,\cdot)\biggr\|^2\\
&=\Vert h \Vert^2-2\sum_{j=1}^n c_j h(x_j)+\sum_{j,k=1}^nc_jc_k K(x_j,x_k).
\end{align*}
The function $g: \R \to \R$ with $g(a)=e(Q_{ac,\X_n},S)^2$ attains its minimum for
\[
 a=\frac{\sum_{j=1}^n c_j h(x_j)}{\sum_{j,k=1}^nc_jc_k K(x_j,x_k)},
\]
where $0/0$ is interpreted as $0$. This yields
\[
e(\X_n,S)^2
=\inf_{c\in \R^n}\inf_{a\in \R} e(Q_{ac,\X_n},S)^2
=\Vert h \Vert^2-\sup_{c\in \R^n}\frac{\left(\sum_{j=1}^n c_j 
h(x_j)\right)^2}{\sum_{j,k=1}^nc_jc_k K(x_j,x_k)}.
\]
The last expression is larger or equal to $\Vert h\Vert^2-\alpha^{-1}$ if, and only if, 
\[
\sum_{j,k=1}^nc_jc_k K(x_j,x_k) \geq \alpha \left(\sum_{j=1}^n c_j h(x_j)\right)^2
\]
holds for all $c\in \R^n$, i.e. when the matrix $\big( K(x_j,x_k) 
- \alpha h(x_j)h(x_k)\big)_{j,k\leq n}$ is positive semi-definite.
This yields the statement.

The proof of the second part follows from the first part 
by taking the infimum over all $\X_n=\{x_1,\dots,x_n\}\subset D.$
\end{proof}

The idea now is to use some properties of the Schur product 
(often also called Hadamard or entrywise product) of matrices.
We denote by $\diag M=(M_{1,1},\dots,M_{n,n})^T$ the diagonal entries of $M$ 
whenever $M\in \R^{n\times n}$. 
Moreover, if $A,B\in\R^{n\times n}$ are two symmetric matrices, 
we write $A\succeq B$ if $A-B$ is positive semi-definite.
The Schur product of $A$ and $B$ 
is the matrix $A\circ B$ with
\[
(A\circ B)_{i,j}=A_{i,j}B_{i,j}\quad\text{for}\quad i,j\le n.
\]
The classical Schur product theorem states that the Schur product of two positive semi-definite matrices
is again positive semi-definite.
However, this statement can be improved \cite{Vy19} when $A=B$.

\begin{prop}\label{prop:main2} 
Let $M\in\R^{n\times n}$ be a positive semi-definite matrix. Then
$$
M\circ M\succeq \frac{1}{n}(\diag M)(\diag M)^T.
$$
\end{prop}
A direct proof of Proposition \ref{prop:main2} may be found in \cite{Vy19}.
As pointed out to the authors by Dmitriy Bilyk, the result follows also from the theory
of positive definite functions on the spheres as developed in the classical work of Schoenberg \cite{Schoen}.
To sketch this approach, let $(\mathbf C_k^{\lambda}(t))_{k=0}^\infty$ denote the sequence of Gegenbauer (or ultraspherical) polynomials.
These are polynomials of order $k$ on $[-1,1]$, which are
orthonormal with respect to the weight \mbox{$(1-t^2)^{\lambda-1/2}$}. Here, $\lambda>-1/2$ is a real parameter.
By the Addition Theorem \cite[Theorem 9.6.3]{Special}, there is a positive constant $C_{k,n}$
and a natural number $c_{k,n}$, both only depending on $k$ and $n$, such that
\begin{equation}\label{eq:Gegen1}
\mathbf C_k^{(n-2)/2}(\langle x,y\rangle)=C_{k,n}\sum_{l=1}^{c_{k,n}} S_{k,l}(x)S_{k,l}(y),\quad x,y\in{\mathbb S}^{n-1}.
\end{equation}
Here, ${\mathbb S}^{n-1}$ stands for the unit sphere in $\R^n$ and $S_{k,1},\dots,S_{k,c_{k,n}}$
form an orthonormal basis of the space of harmonic polynomials of degree $k$ in $\R^n$.

If now $X=(x_{i,j})_{i,j=1}^n\in\R^{n\times n}$ is a positive semi-definite matrix with ones on the diagonal
and $f(t)=\sum_{k=0}^\infty a_k \mathbf C_{k}^{(n-2)/2}(t)$ with $a_k\ge 0$, then
$(f(x_{i,j}))_{i,j=1}^n$ is also positive semi-definite. Indeed, we can write $x_{i,j}=\langle x_i,x_j\rangle$ for some vectors
$x_1,\dots,x_n\in{\mathbb S}^{n-1}$ and use \eqref{eq:Gegen1} to compute for every $c\in\R^n$
\begin{align*}
\sum_{i,j=1}^n c_ic_j f(x_{i,j})&=\sum_{i,j=1}^n c_ic_j\sum_{k=0}^\infty a_k \mathbf C_k^{(n-2)/2}(\langle x_i,x_j\rangle)\\
&=C_{k,n}\sum_{i,j=1}^n c_ic_j\sum_{k=0}^\infty a_k \sum_{l=1}^{c_{k,n}} S_{k,l}(x_i)S_{k,l}(x_j)\\
&=C_{k,n}\sum_{k=0}^\infty a_k \sum_{l=1}^{c_{k,n}} \Bigl(\sum_{i=1}^n c_i S_{k,l}(x_i)\Bigr)^2\ge 0.
\end{align*}
For positive semi-definite matrices $M\in\R^{n\times n}$ with ones on the diagonal, Proposition \ref{prop:main2} then follows by
observing that $f(t)=t^2-\frac{1}{n}$ is (up to a positive multiplicative constant) exactly the polynomial $\mathbf C_2^{(n-2)/2}(t).$
Finally, the general form of Proposition \ref{prop:main2} is given by a simple scaling argument.
$\hfill \square$

\section{Homogeneous tensor products}\label{sec:homo}

We now use 
Propositions~\ref{prop:error_vs_definiteness}
and \ref{prop:main2} in order to obtain 
the curse of dimensionality for certain tensor product (integration)
problems.
In this section, we consider homogeneous tensor products, i.e.,
$\mathbf F_d = F_1\otimes\cdots\otimes F_1$ and $\mathbf h_d=h_1\otimes\cdots\otimes h_1$.

\begin{thm}
\label{thm:curse_tensor_product}
Let $F_1$ be a RKHS on $D_1$. 
Assume that there are functions $e_1$ and $e_2$ on $D_1$
such that $e_1^2, e_2^2$ and $\sqrt{2} e_1 e_2$ 
are orthonormal in $F_1$ and let
$S_1$ be a linear functional with 
$S_1(e_i^2)=\sqrt 2 /2$
and $S_1(e_1 e_2)=0$.
Then the tensor product problem $\mathbf S_d$ 
satisfies
\[
	e(n,\mathbf S_d)^2 \geq 1 - n \,  2^{-d} .
 \]
In particular, it suffers from the curse of dimensionality
since we need at least $2^d \, (1-\varepsilon^2 )$ 
function values to achieve the error $\varepsilon$.   
\end{thm}

\begin{rem}
We usually work with normalized problems, i.e., 
we assume that $e(0,S_1) = \Vert  h_1 \Vert = 1$
and thus $e(0,\mathbf S_d) = \Vert\mathbf  h_d \Vert = 1$.
Note that the only normalized functional $S_1$ that satisfies 
the above conditions is given by 
$S_1=\langle \cdot, h_1 \rangle$ 
with the representer $h_1= \frac{1}{2} \sqrt{2} ( e_1^2 + e_2^2)$.
\end{rem}

\begin{proof}
Without loss of generality, we may assume that $F_1$ is 3-dimensional, i.e.,
that $e_1^2$, $e_2^2$ and $\sqrt{2}e_1e_2$ form an orthonormal basis.
Then also $b_1 %
=  \frac{1}{2} \sqrt{2} (e_1^2 + e_2^2)$, 
$b_2= \frac{1}{2} \sqrt{2} (e_1^2 - e_2^2)$, 
and $b_3 = \sqrt{2} e_1 e_2$ 
is an orthonormal basis since $b_1$ and $b_2$ are just a rotation of $e_1^2$ and $e_2^2$.
On the 3-dimensional space $F_1$, the functional $S_1$ is represented by $h_1=b_1$.
The function 
\[
M_1\colon D_1\times D_1 \to \R, \quad
M_1(x,y)= \sum_{i=1}^2 e_i(x) e_i(y),
\] 
is a reproducing kernel on $D_1$.
The reproducing kernel $K_1$ of $F_1$ satisfies
\[
K_1(x,y) = \sum_{i=1}^3 b_i(x) b_i(y) = \left( \sum_{i=1}^2 e_i(x) e_i(y) \right)^2 = M_1(x,y)^2
\]
for all $x,y\in D_1$. Moreover, we have 
$h_1(x) = \frac{1}{2} \sqrt{2}  M_1(x,x)$ for all $x\in D_1$.
Therefore, also $\mathbf K_d(x,y)= M_d(x,y)^2$ and 
	$\mathbf h_d(x)=  2^{-d/2} \,\mathbf  M_d(x,x)$ for $x,y\in \mathbf D_d$,
where $\mathbf M_d$ is the $d$-times tensor product of $M_1$
and $\mathbf h_d$ is the $d$-times tensor product of $h_1$. 
By Proposition~\ref{prop:main2} 
the matrix 
\[
\big(\mathbf  K_d(x_j,x_k) - n^{-1} \, 2^d \, 
 \mathbf h_d(x_j)\mathbf h_d(x_k)\big)_{j,k\leq n}
\]
is positive semi-definite for all $x_1,\hdots,x_n\in \mathbf D_d$. 
Proposition~\ref{prop:error_vs_definiteness} yields that
\[
	e(n,\mathbf S_d)^2 \geq 1 - n \, 2^{-d} .
\]
\end{proof}

We now consider several applications of this result
and start with a general remark. As in the proof of 
Theorem~\ref{thm:curse_tensor_product}
these examples are finite-dimensional, 
$F_1$ has dimension three. 
Of course the lower bounds are valid for all larger Hilbert spaces 
with the same norm on the subspace from $F_1$. 
A RKHS is equivalently given by the scalar product or the kernel 
or a complete orthonormal system. 
Unfortunately, there is no simple way to compute the scalar product 
if the kernel is given, or vice versa. 
Because of the form of our result, it is convenient to start with 
three vectors $b_i$ from $F_1$ and to claim that they are orthonormal. 
This \emph{defines} the scalar product (and the kernel) for this three-dimensional 
space, though it is possible to extend the scalar product to larger spaces 
in many different ways. 
Because of the very specific form of the orthonormal system that is required in our result, 
we are \emph{not} free to choose the scalar product and hence, 
later, will have to work with Sobolev spaces with a non-standard 
norm or scalar product.

This means that all examples of this section will be
specified only by defining two functions $e_1$ and $e_2$.
These functions immediately define a 3-dimensional space $F_1$
(with orthonormal basis $e_1^2$, $e_2^2$ and $\sqrt 2 e_1 e_2$
or equivalently $b_1$, $b_2$ and $b_3$ as above)
and a linear functional $S_1$ with representer $h_1=b_1$ on $F_1$.
There are always 
many ways of writing down the norm of $F_1$
(and thus many ways to interpret $F_1$ as a subspace of some larger space)
and we will provide some of them.
But since this is just for the purpose of interpretation,
we will not always provide the (tedious) computations.

\subsection{Trigonometric polynomials of degree 1}\label{sec:trig}

This example is already contained in  Vyb\'\i ral~\cite{Vy19}; 
now we can see it as an application 
of the general Theorem~\ref{thm:curse_tensor_product}. 
It is defined by the choice
$e_1(x)= 2^{1/4} \cos(\pi x)$ 
and $e_2(x)=2^{1/4} \sin(\pi x)$ on $[0,1]$. 
Then one obtains
$b_1 = h_1 = 1$
and
$b_2 (x) = \cos (2 \pi x)$ and
$b_3 (x) = \sin (2 \pi x)$.
The functions $b_i$ are orthonormal by definition
and we have many ways to define 
matching norms on larger spaces. 
One way to write down the norm
on the space $F_1$ of
trigonometric polynomials of degree 
1 on the interval $[0,1]$ is  
\[
\Vert f \Vert^2 = \Vert f \Vert_2^2 + \frac{1}{4 \pi^2} \Vert f'\Vert_2^2 .
\]
One only has to check that the $b_i$'s indeed are orthonormal with 
respect to this Sobolev Hilbert space.
For $d \in \N$ we take the tensor product space 
of the three-dimensional $F_1$
with the kernel 
\[
\mathbf K_d (x,y) = \prod_{i=1}^d (1+ \cos ( 2\pi  (x_i-y_i)) ) .
\] 
We obtain 
$
\Gamma_d  = \sup_x \mathbf K_d (x,x)^{1/2} = 2^{d/2}  
$
and $\Gamma_d$ is the norm 
of the embedding of $\mathbf F_d$ into the 
space of
continuous functions 
with the sup norm, see Lemma~5 of~\cite{NUWZ18}.
Hence functions in the unit ball of $\mathbf F_d$ may take large values 
if $d$ is large, but the integral is bounded by one. 
By applying Theorem~\ref{thm:curse_tensor_product} 
we obtain the following result of \cite{Vy19} that 
solved an
open problem of 
\cite{Erich}, see also \cite{HV}. 

\begin{cor}
\label{cor1} 
Let $F_1$ be the RKHS on $[0,1]$ with the orthonormal system
$1$, $\cos (2\pi x)$ and $\sin (2\pi x)$.  
Then the integration problem $\mathbf S_d=\langle \cdot, 1 \rangle$ 
on the tensor product space $\mathbf F_d$ satisfies
\[
	e(n,\mathbf S_d)^2 \geq 1 - n \,  2^{-d} .
 \]
In particular, it suffers from the curse of dimensionality.  
\end{cor}

\begin{rem} 
The same vector space with dimension $3^d$ was studied  
by Sloan and Wo\'zniakowski~\cite{SW97} who 
were mainly interested in the Korobov class $E_{\alpha, d}$ given by 
all functions with Fourier coefficients 
$\hat f (h)$ such that 
$$
| \hat f (h) | \le  \prod_{i=1}^d (\max (1, |h_i|))^{-\alpha},
$$
where $\alpha$ might be large. 
The authors of \cite{SW97} proved that the optimal worst case error 
for this class is 1 for $n<2^d$.
This holds for the whole Korobov class 
and also for the subset of trigonometric polynomials of degree 1
(which is larger than the unit ball 
considered in Corollary~\ref{cor1}). 
For these polynomials 
the error is zero for $n=2^d$, since a product rule 
with $2^d$ evaluations is exact. 

As a by-product the authors of \cite{SW97} also obtain the fact that exactly $2^d$ 
function values are needed to obtain an exact quadrature formula for 
this space of polynomials. 
This last property also follows from Corollary~\ref{cor1}
and is of course independent of the used norm.  
\end{rem} 

\subsection{Gaussian integration for polynomials of degree 2}

Let $F_1$ be the space of polynomials on $\R$ with degree at most 2, equipped with the scalar product
\[
 \langle f, g \rangle = f(0)g(0) +  \frac12 f'(0)g'(0) + \frac14 \int_\R f''(x) g''(x)\,{\rm d}\mu_1(x),
\]
where $\mu_1$ is the standard Gaussian measure on $\R$.
We consider the integration problem
\[
 S_1\colon F_1 \to \R, \quad S_1(f)=\int_\R f(x)\,{\rm d}\mu_1(x).
\] 
The tensor product problem for $d\in\N$ is given by the functional
\[
 \mathbf S_d\colon F_d \to \R,\quad  \mathbf S_d(f)=\int_{\R^d} f(x)\,{\rm d}\mu_d(x),
\] 
on the tensor product space $\mathbf F_d$,
which consists of all $d$-variate polynomials of 
degree 2 or less in every variable.
Here, $\mu_d$ is the standard Gaussian measure on $\R^d$.
By Theorem~\ref{thm:curse_tensor_product}, 
this problem suffers from the curse of dimensionality.
To see this, it is enough to choose $e_1(x)=1$ and $e_2(x)=x$.
We observe that the functions
$e_1^2=1$, $e_2^2=x^2$ and $\sqrt 2 e_1e_2=\sqrt 2 x$
are orthonormal in $F_1$
and that $S_1(e_i^2)=1$ and $S_1(e_1 e_2)=0$.
For this, note that $S_1(x^j)$ is the $j$th 
moment of a standard Gaussian variable.
In particular, $e(0,S_1)=\sqrt{2}$.
Thus, the functional $S_1'=2^{-1/2} S_1$ satisfies the 
conditions of Theorem \ref{thm:curse_tensor_product} 
and an application of the theorem immediately yields the following.

\begin{cor} 
	Take the RKHS $F_1$ on $\R$ which is generated by 
	the  orthonormal system 
	$1$, $x^2$ and $\sqrt{2} x$. 
	Then the problem $\mathbf S_d(f)=\int_{\R^d} f(x)\,d\mu_d(x)$ of Gaussian integration on
	the tensor product space $\mathbf F_d$ satisfies
	$$
	\frac{e(n,\mathbf S_d)^2}{e(0,\mathbf S_d)^2} \ge 1 - n \, 2^{-d}.
	$$
	In particular, we obtain the curse of dimensionality 
	and the fact that exactly $n=2^d$ function values 
	are needed for exact integration. 
\end{cor}

For the last statement it is enough to consider product Gauss 
formulas with $n=2^d$ function values that are exact for all 
polynomials from $\mathbf F_d$. 

\subsection{Integration for polynomials of degree 2 on $[-\frac{1}{2},\frac{1}{2}]$}

Let $F_1$ be the space of polynomials on $\R$ with degree at most 2,
defined on an interval of unit length. 
For convenience and symmetry we take the interval $[-1/2, 1/2]$. 
The univariate problem is given by 
$S_1(f) = \int_{-1/2}^{1/2} f(x) \,{\rm d} x$
and again we want to apply Theorem~\ref{thm:curse_tensor_product}.
For our construction we need 
$S_1(e^2_i) = \frac{1}{2} \sqrt{2}$
and $S_1(e_1e_2)=0$. 
This is achieved with the choice
$e_1= 2^{-1/4} $ and 
$e_2(x) = 72^{1/4} x$.
If we apply 
Theorem~\ref{thm:curse_tensor_product} 
then we obtain the following.

\begin{cor} 
	Take the RKHS $F_1$ on $I=[-1/2,1/2]$ which is generated by 
	the  orthonormal system 
	$\frac{1}{2} \sqrt{2}$, $\sqrt{72} x^2$ and $\sqrt{12} x$. 
	Then the integration problem $\mathbf S_d(f)=\int_{I^d} f(x)\, {\rm d} x$ on $\mathbf F_d$ satisfies
	$$
	e(n,\mathbf S_d)^2 \ge 1 - n \, 2^{-d}.
	$$
	In particular, we obtain the curse of dimensionality 
	and the fact that exactly $n=2^d$ function values 
	are needed for the exact integration. 
\end{cor}

For the last statement it is again enough to consider product Gauss 
formulas with $n=2^d$ function values that are exact for all 
polynomials from~$\mathbf F_d$.

Observe that we are forced by our approach to take this norm
on $F_1$ and on the tensor product space. 
We admit that the norm on $F_1$ is not 
a standard Sobolev norm, actually it looks rather artificial.  
The norm in $F_1$ is a weighted $\ell_2$-norm of Taylor coefficients. 
For $d=1$ and $f(x) = ax^2 + bx +c$ we obtain the norm 
$$
\Vert f \Vert^2 = \frac{a^2}{72} + \frac{b^2}{12} + 2 c^2 . 
$$
It can also be written in the form 
\begin{align*}
\Vert f \Vert^2 &= 2f(0)^2 + \frac{1}{12} f'(0)^2 + \frac{1}{288} f''(0)^2\\
&= 2f(0)^2 + \frac{1}{12} f'(0)^2 + \frac{1}{288} \int_{-1/2}^{1/2} f''(x)^2 \, {\rm d} x
\end{align*}
and so looks at least a little like a Sobolev norm. 
By our approach, we are not free to choose the norm 
and obtain lower bounds only for very specific norms. 
For the given norm we obtain 
$$
\Gamma = \sup_{x\in I} K_1(x,x)^{1/2} = 8^{1/2} 
$$
and $\Gamma^d$ is the norm 
of the embedding of $\mathbf F_d$ %
into the 
space of
continuous functions 
with the sup norm. 
Hence functions in the unit ball of $\mathbf F_d$ %
may take large values 
if $d$ is large, but the integral is bounded by one. 
 
\subsection{Integration of functions with zero boundary conditions}

As another application of Theorem~\ref{thm:curse_tensor_product}, we consider the integration of
smooth functions with zero on the boundary.
For that sake, let $e_1(x)=2^{1/4}\sin(\pi x)$ and $e_2(x)=2^{1/4}\sin(2\pi x)$ for $x\in[0,1]$.
Further, let $F_1$ be a three-dimensional space of functions on $[0,1]$, such that the system
\begin{align}
\notag e_1^2(x)&=\sqrt{2}\sin^2(\pi x),\\
\label{eq:zero_bound} e_2^2(x)&=\sqrt{2}\sin^2(2\pi x),\\
\notag \sqrt{2}e_1(x)e_2(x)&=2\sin(\pi x)\sin(2\pi x)
\end{align}
forms an orthonormal basis of $F_1$.
The norm on $F_1$ is uniquely determined and it can be expressed
for example by
$$
\|f\|^2=\frac{1}{3}f(1/2)^2+\frac{1}{12\pi^2}\int_0^1 (1+4\sin^2(2\pi x)) f'(x)^2dx
$$
or
$$
\|f\|^2=\frac{1}{2}f(1/2)^2+\frac{1}{16\pi^2}f'(1/2)^2+\frac{1}{128\pi^4}[f''(1/4)+f''(3/4)]^2,
$$
which coincide on $F_1$.
We consider the integration problem on $F_1$ defined by
$$
S_1:F_1\to \R,\quad S_1(f)=\int_0^1 f(x)\,{\rm d}x
$$
and its tensor product version $\mathbf S_d$ on $\mathbf F_d$. 
We observe that $S_1(e_1^2)=S_1(e_2^2)=\sqrt{2}/2$ and $S_1(e_1e_2)=0$. 

\begin{cor} Let $F_1$ be a three-dimensional RKHS on $[0,1]$ such that the functions in \eqref{eq:zero_bound}
form its orthonormal basis. 
Then the integration problem $\mathbf S_d(f)=\int_{[0,1]^d}f(x)\, {\rm d}x$ satisfies
$$
e(n,\mathbf S_d)^2\ge 1-n\,2^{-d},
$$
i.e. it suffers from the curse of dimensionality.
\end{cor}
\begin{rem}
Let us observe that every $f\in F_1$ satisfies $f(0)=f(1)=f'(0)=f'(1)=0$. 
This means that the functions from $F_d$ 
and all their partial derivatives of order at most one in any of the variables
vanish on the boundary of the unit cube.
\end{rem}

\subsection{Hilbert spaces with decomposable kernels}

Another known method to prove lower bounds for tensor product func\-tion\-als works for so 
called decomposable kernels and slight modifications, see \cite[Chapter 11]{NW10}. 
There is some intersection where our method and the decomposable kernel method both work. 

Let $F_1$ be a RKHS on $D_1 \subset \R$ with reproducing kernel $K_1$.
The kernel $K_1$ is called decomposable if there exists $a^*\in \R$ such that the sets
$$
 D_{(1)} = \{ x \in D_1 \mid x\le a^*\} 
\quad \text{and} \quad
D_{(2)} = \{ x \in D_1 \mid x\ge a^*\} 
$$ 
are nonempty and $K_1(x,y)=0$ 
if $(x,y) \in  D_{(1)} \times D_{(2)}$ or $(x,y) \in  D_{(2)} \times D_{(1)}$.
If $K_1$ is decomposable, then $F_1$ is an orthonormal sum of $F_{(1)}$ and $F_{(2)}$ 
consisting of the functions in $F_1$ with support in $D_{(1)}$ and $D_{(2)}$, respectively.

Choosing now arbitrary suitably scaled functions 
$e_1$ with support in $D_{(1)}$ and $e_2$ with support in $D_{(2)}$
such that $e_1^2 \in F_{(1)}$ and $e_2^2 \in F_{(2)}$,
we automatically have that $e_1^2$ and $e_2^2$ are orthonormal in $F_1$ and $e_1 e_2 =0$. 
The proof of Theorem 3 is easily adapted to this case and gives the next corollary.

\begin{cor} \label{cor:_8}
Let $F_1$ be a RKHS on $D_1 \subset \R$ with decomposable reproducing kernel.
Let $e_1$ and $e_2$ be as above  
and let $h_1= \frac{1}{2} \sqrt{2} ( e_1^2 + e_2^2)$. 
Then the tensor product problem $\mathbf S_d=\langle \cdot, \mathbf h_d\rangle$ satisfies
\[
	e(n,\mathbf S_d)^2 \geq 1 - n \,  2^{-d} .
 \]
In particular, it suffers from the curse of dimensionality.  
\end{cor}

One particular example, where this corollary is applicable, is the centered $L_2$-discrepancy. 
Here $F_1$ consists of absolutely continuous functions $f$ on $[0,1]$ with $f(1/2)=0$ and $f' \in L_2[0,1]$.
The norm of $f$ in $F_1$ is the $L_2$-norm of $f'$. The kernel of $F_1$ is given by $K_1(x,y)=\big(|x-1/2| + |y-1/2| - |x-y|\big)/2$
and is decomposable with respect to $D_{(1)}=[0,1/2]$ and $D_{(2)}=[1/2,1]$.
The normalized representer of the integration problem is $h_1(x)=\big(|x-1/2| - |x-1/2|^2\big)/2$. Then $e_1^2$ is the normalized
restriction of $h_1$ to the interval $[0,1/2]$, similarly, $e_2^2$ is the normalized restriction of $h_1$ to the interval $[1/2,1]$.
Since $h_1$ is nonnegative, such functions $e_1$ and $e_2$ exist.

Corollary \ref{cor:_8} is a special case (for $\alpha=1/2$) of \cite[ Theorem 11.8]{NW10}. 
As such, it will not give any new results. Nevertheless, it seems appropriate to note the connection. 
It would be interesting to know if the full strength of \cite[ Theorem 11.8]{NW10} can be 
obtained via this approach or the variants described in the next section. 

\subsection{Exact Integration}

In all the examples from above, we obtained that
$$
e(2^d -1, \mathbf S_d) > 0,
$$
so that it is not possible to compute the integral exactly with less than $2^d$ function values.
One may ask whether this is the case for all nontrivial tensor product problems.
Here a problem is called trivial if $e(1, S_1)=0$. 
Then we have also $e(1,\mathbf  S_d)=0$ for all $d$. 
In general, the answer is ``no", examples with 
$e(d,\mathbf S_d)>0$   but $e(d+1,\mathbf  S_d) = 0$ can be found in
\cite[Section 11.3]{NW10} 
which is based on~\cite{NSW97}. 
However, we obtain the following criterion
under which the answer to the above question is ``yes".

\begin{cor}
If there are functions $e_1$ and $e_2$ 
such that $e_1^2, e_2^2, e_1e_2 \in F_1$ are linearly independent with 
$S_1(e^2_i) \not= 0$ and $S_1(e_1 e_2)=0$, then 
$$
e(2^d-1,\mathbf  S_d ) > 0.
$$
\end{cor}

This follows from Theorem~\ref{thm:curse_tensor_product}
since the statement on exact integration does not depend on the
norm (resp.\ scalar product) of $F_1$. We may simply
apply the theorem
to the 3-dimensional space which is defined by the orthonormal basis
$\tilde e_1^2, \tilde e_2^2$ and $\sqrt 2 \tilde e_1\tilde e_2$, 
where $\tilde e_i := c_i e_i$
with $c_i \in \R$ such that $S_1(\tilde e_i^2)=\sqrt 2/2$.

\section{Non-homogeneous tensor products}
\label{sec:non-homo}

We now turn to tensor products whose factors $F_i$ and $h_i$ may be different for each $i\le d$.
We start with the following generalization of Theorem \ref{thm:curse_tensor_product}, which involves an additional parameter $\alpha_i$.

\begin{thm}\label{thm:abstract-trigonometric-weighted}
For all $i\le d$, let $F_i$ be a RKHS and let $S_i$ be a bounded linear functional on $F_i$ 
with unit norm and nonnegative representer~$h_i$. 
Assume that there are functions $f_i$ and $g_i$ in $F_i$ and a number $\alpha_i\in (0,1]$ such that
$(h_i,f_i,g_i)$ is orthonormal in $F_i$ and $\alpha_i h_i=\sqrt{f_i^2+g_i^2}$.
Then the tensor product problem $\mathbf S_d=S_1\otimes\hdots\otimes S_d$ %
satisfies for all $n\in\N$ that
\[
	e(n,\mathbf S_d)^2 \ge 1-n\prod_{i=1}^d(1+\alpha_i^2)^{-1}.
\]
\end{thm}

\begin{proof}
 Let $D_i$ be the domain of the space $F_i$.
 Without loss of generality, we may assume that $(h_i,f_i,g_i)$ is an
 orthonormal basis of $F_i$. In this case, the reproducing kernel of $F_i$ is
 given by
 \[
  K_i: D_i\times D_i \to \R, \quad K_i(x,y)=h_i(x)h_i(y)+f_i(x)f_i(y)+g_i(x)g_i(y).
 \]
 Let us consider the functions
 \[
  a_i=2^{-1/4} \sqrt{\alpha_i h_i + f_i},
  \qquad
  b_i=2^{-1/4} {\rm sgn}(g_i) \sqrt{\alpha_i h_i - f_i}
 \]
 on the domain $D_i$ of $F_i$.
 These functions are well defined since $\alpha_i h_i\ge \vert f_i\vert$
 and linearly independent since $h$ and $f$ are linearly independent.
 The function
 \[
  M_i: D_i\times D_i \to \R, \quad M_i(x,y)=a_i(x)a_i(y)+b_i(x)b_i(y)
 \]
 is a reproducing kernel on $D_i$ and its diagonal is $\sqrt{2} \alpha_i h_i$. 
 A simple computation shows for all $x,y\in D_i$ that
 \[
  K_i(x,y)=M_i^2(x,y)+(1-\alpha_i^2) h_i(x)h_i(y).
 \]
 Let now $\mathbf K_d$ be the reproducing kernel
 of the product space $\mathbf F_d=F_1\otimes\hdots\otimes F_d$ 
 with domain $\mathbf D_d = D_1 \times\hdots\times D_d$
 and let $x_1,\hdots,x_n\in \mathbf D_d$.
 We have
 \[
  \mathbf K_d(x_j,x_k) =\prod_{i=1}^d K_i(x_{j,i},x_{k,i}) =\sum_{A\subset \{1,\dots,d\}} \mathbf K_d^A(x_j,x_k),
 \]
 where
 \[
  \mathbf K_d^A(x_j,x_k) = \prod_{i\in A} M_i^2(x_{j,i},x_{k,i}) \, \prod_{i\not\in A}(1-\alpha_i^2)h_i(x_{j,i})h_i(x_{k,i}).
 \]
The application of Proposition \ref{prop:main2} yields
$$
\Biggl(\prod_{i\in A} M_i^2(x_{j,i},x_{k,i})\Biggr)_{j,k=1}^n\succeq
\frac{1}{n}\Biggl(\prod_{i\in A} 2\alpha_i^2h_i(x_{j,i})h_i(x_{k,i})\Biggr)_{j,k=1}^n
$$
and hence
\[
\Bigl(\mathbf K_d^A(x_j,x_k)\Bigr)_{j,k=1}^n
\succeq 
\frac{1}{n}\prod_{i\in A} 2\alpha_i^2 \prod_{i\not\in A}(1-\alpha_i^2)\,
\Bigl(\mathbf{h}_d(x_j)\mathbf{h}_d(x_k)\Bigr)_{j,k=1}^n,
\]
where $\mathbf h_d=h_1\otimes\hdots\otimes h_d$ is the representer of 
the product functional~$\mathbf S_d$.
Summing over all subsets $A$, we arrive at
\begin{align*}
\Bigl(\mathbf K_d(x_j,x_k)\Bigr)_{j,k=1}^n&=\sum_{A\subset\{1,\dots,d\}}\Bigl(\mathbf K_d^A(x_j,x_k)\Bigr)_{j,k=1}^n\\
&\succeq\frac{1}{n}\sum_{A\subset\{1,\dots,d\}}\prod_{i\in A} 2\alpha_i^2 \prod_{i\not\in A}(1-\alpha_i^2)\,
\Bigl(\mathbf{h}_d(x_j)\mathbf{h}_d(x_k)\Bigr)_{j,k=1}^n\\
&=\frac{1}{n}\prod_{i=1}^d (1+\alpha_i^2)\,
\Bigl(\mathbf{h}_d(x_j)\mathbf{h}_d(x_k)\Bigr)_{j,k=1}^n.
\end{align*}
Now the statement follows by Proposition \ref{prop:error_vs_definiteness}.
\end{proof}

As applications of this result, 
we consider spaces of trigonometric polynomials,
Korobov spaces with increasing smoothness
and Korobov spaces with product weights.

\subsection{Trigonometric polynomials}
\label{sec:increasing-smoothness}

The most prominent special case of Theorem \ref{thm:abstract-trigonometric-weighted}
is the case of trigonometric polynomials of order at most one, i.e.,
\begin{equation}\label{eq:hfg}
h_i(x)=1,\quad f_i(x)=\alpha_i\cos(2\pi x),\quad g_i(x)=\alpha_i\sin(2\pi x),\quad x\in[0,1],
\end{equation}
which leads to the following result.
\begin{cor}\label{thm:alpha}
For all $1\le i \le d$, let  $\alpha_i\in(0,1]$ and let $F_i$ be a RKHS on $[0,1]$ such that $(h_i,f_i,g_i)$ defined in \eqref{eq:hfg}
are orthonormal in $F_i$. Then the integration problem $\mathbf S_d(f)=\int_{[0,1]^d}f(x){\rm d}x$ satisfies on ${\mathbf F_d}=F_1\otimes\dots\otimes F_d$
\[
e(n,\mathbf S_d)^2\ge 1-n\prod_{i=1}^d(1+\alpha_i^2)^{-1}.
\]
\end{cor}

Corollary~\ref{thm:alpha} can be used to prove lower bounds for numerical integration on spaces with varying smoothness.
Such classes were studied in \cite{KSU20,PW10} 
for the approximation problem and upper bounds for numerical integration were provided
in \cite[Section 10.7.4]{NW10}. We first recall the notation. %

For a non-decreasing sequence of positive integers $r=(r_i)_{i=1}^{\infty}$
we consider the spaces $H_{1,r_i}$ of 1-periodic real valued functions $f$ defined on $[0,1]$ such that $f^{(r_i-1)}$
is absolutely continuous and $f^{(r_i)}$ belongs to $L_2([0,1])$. The norm on $H_{1,r_i}$ is given by
\[
\|f\|^2_{H_{1,r_i}}=\Bigl|\int_0^1 f(x){\rm d}x\Bigr|^2+\int_0^1|f^{(r_i)}(x)|^2{\rm d}x.
\]
The Korobov space of varying smoothness is then defined by
\[
{\mathbf F}_d=H_{1,r_1}\otimes\dots\otimes H_{1,r_d}.
\]

If we set $\alpha_i=\sqrt{2}\cdot(2\pi)^{-r_i}$, then $(h_i,f_i,g_i)$ from \eqref{eq:hfg}
form an orthonormal system in $H_{1,r_i}$ and we denote their span in $H_{1,r_i}$ by $\widetilde H_{1,r_i}$.
We will prove lower bounds for ${\mathbf F}_d$ by actually considering only the $3^d$-dimensional space
\[
\widetilde {\mathbf F}_d=\widetilde H_{1,r_1}\otimes\dots\otimes\widetilde H_{1,r_d}.
\]
We consider the integration problem
\[
\mathbf S_d(f)=\int_{[0,1]^d}f(x){\rm d}x,\quad f\in \widetilde {\mathbf F}_d.
\]
We call the problem \emph{polynomially tractable} if there are positive constants $C,p,q>0$ such that
\[
 n(\varepsilon,\mathbf S_d) \le C d^p \varepsilon^{-q}
\]
for all $\varepsilon>0$ and $d\in\N$.
We call it \emph{strongly polynomially tractable} if we can choose $p=0$ in this estimate.
Moreover, the problem is called \emph{weakly tractable} if
\[
 \lim_{\varepsilon^{-1}+d \to \infty} \frac{\ln n(\varepsilon,\mathbf S_d)}{\varepsilon^{-1}+d} = 0.
\]
It was observed in \cite[Section 10.7.4]{NW10} (see also Corollary 10.5 there), that 
\begin{itemize}
\item if $L^{{\rm sup}}:=\limsup_{i\to\infty}\frac{\ln (i)}{r_i}<2\ln (2\pi)$, then integration on $\mathbf F_d$ is strongly polynomially tractable;
\item if $L^{{\rm sup}}<+\infty$, then integration on $\mathbf F_d$ is weakly tractable.
\end{itemize}

We complement this by showing lower bounds for numerical integration on $\widetilde {\mathbf F}_d$ (which of course also apply to the larger space ${\mathbf F}_d$).
By Corollary~\ref{thm:alpha}, we obtain the estimate
\begin{equation}\label{eq:hfg2}
e(n,\mathbf S_d)^2\ge 1-n\prod_{i=1}^d(1+\alpha_i^2)^{-1}=1-n\prod_{i=1}^d (1+2\cdot(2\pi)^{-2r_i})^{-1}.
\end{equation}

\begin{cor} For $d\ge 2$, let ${\mathbf F}_d$ be the Korobov space of varying smoothness on $[0,1]^d$ given by the sequence $r=(r_i)_{i=1}^\infty$
and let $\widetilde {\mathbf F}_d$ be its $3^d$-dimensional subspace of trigonometric polynomials of order at most one in each variable. 
\begin{enumerate}
\item[(i)] If $r=(r_i)_{i=1}^\infty$ is bounded, then numerical integration on $\widetilde {\mathbf F}_d$ (and hence also on ${\mathbf F}_d$) suffers from the curse of dimension.
\item[(ii)] If $L^{{\rm inf}}:=\liminf_{i\to\infty}\frac{\ln i}{r_i}=\infty$, then numerical integration on $\widetilde {\mathbf F}_d$
(and hence also on ${\mathbf F}_d$) satisfies for any $\varepsilon,\beta>0$ that
\[
 n(\varepsilon,\mathbf S_d) \,\ge\, c_{\varepsilon,\beta} \, \exp\left(d^{1-\beta}\right).
\]%
\item[(iii)] If $L^{{\rm inf}}>2\ln(2\pi)$, then numerical integration on $\widetilde {\mathbf F}_d$ (and hence also on ${\mathbf F}_d$)
is not polynomially tractable.
\end{enumerate}
\end{cor}

\begin{proof}
The proof is a direct consequence of \eqref{eq:hfg2}. If $r_i\le R<\infty$ for all $i\in\N$, then
\[
e(n,\mathbf S_d)^2\ge 1-n(1+2\cdot (2\pi)^{-2R})^{-d}.
\]
This implies $n(\varepsilon,\mathbf S_d)\ge (1-\varepsilon^2)(1+2\cdot (2\pi)^{-2R})^d$ and finishes the proof of (i).

To prove (ii) and (iii), we observe that there is some $0<\beta<1$ and $i_0\in\N$ such that $2r_i\ln(2\pi)\le \beta\ln(i)$ for $i\ge i_0.$
In the case of (ii) we can even find such $i_0=i_0(\beta)$ for any $0<\beta<1$.
Consequently, for $d$ large enough,
\begin{align*}
n(\varepsilon,\mathbf S_d)&\ge (1-\varepsilon^2)\prod_{i=1}^d(1+2(2\pi)^{-2r_i})\ge 
(1-\varepsilon^2)\prod_{i=i_0}^d(1+2i^{-\beta})\\&\ge  (1-\varepsilon^2)\prod_{i=i_0}^d\exp(i^{-\beta})
=(1-\varepsilon^2)\exp\Bigl(\sum_{i=i_0}^d i^{-\beta}\Bigr)\\
&\ge (1-\varepsilon^2)\exp(c_\beta d^{1-\beta}),
\end{align*}
which shows both (ii) and (iii).
\end{proof}

\subsection{Korobov spaces with product weights}

In a quite similar manner, Corollary~\ref{thm:alpha} can be used to re-prove the lower bounds for numerical integration
on Korobov spaces with product weights, see 
\cite{HW01} or \cite[Section 16.8]{NW10}. 
Again, we first recall the necessary notation, see \cite[Appendix A]{NW08} for details.
For a real parameter $s>1/2$, we define
\[
\varrho_{1,s,\gamma}(h)=\begin{cases}1,\quad &h=0,\\\displaystyle\frac{|2\pi h|^{2s}}{\gamma},&h\in\Z\setminus\{0\}.\end{cases}
\]
The space $H_{1,s,\gamma}$ of square-integrable functions on $[0,1]$ is defined by the norm
\[
\|f\|^2_{H_{1,s,\gamma}}=\sum_{h\in\Z}\varrho_{1,s,\gamma}(h)|\hat f(h)|^2,
\]
where
\[
\hat f(h)=\int_0^1 \exp(-2\pi {\mathrm i}\, hx)f(x){\rm d}x,\quad h\in\Z
\]
are the Fourier coefficients of $f$ and ${\mathrm i}=\sqrt{-1}$ is the imaginary unit.

If $\gamma=(\gamma_{d,j})_{d\in\N, 1\le j\le d}$ is a sequence of positive weights, 
the weighted Korobov space (with product weights $\gamma$)  $H_{d,s,\gamma}$ is defined as
the tensor product
\[
H_{d,s,\gamma}=H_{1,s,\gamma_{d,1}}\otimes \dots \otimes H_{1,s,\gamma_{d,d}}.
\]
If $\alpha_{d,j}=\sqrt{2\gamma_{d,j}}\cdot (2\pi)^{-s}$, the functions $(1,\alpha_{d,j}\cos(2\pi x),\alpha_{d,j}\sin(2\pi x))$ are orthonormal in $H_{1,s,\gamma_{d,j}}$.
We denote their linear span in $H_{1,s,\gamma_{d,j}}$ by $\widetilde H_{1,s,\gamma_{d,j}}$ and
\[
\widetilde H_{d,s,\gamma}=\widetilde H_{1,s,\gamma_{d,1}}\otimes \dots \otimes \widetilde H_{1,s,\gamma_{d,d}}.
\]

Using Corollary~\ref{thm:alpha}, we can re-prove (in a rather straightforward way) the lower bounds of Theorem 16.16 in \cite{NW10}.
Moreover, we show that the same lower bounds apply also to the much smaller subspaces $\widetilde H_{d,s,\gamma}$.
\begin{prop}
Let $\mathbf S_d(f)=\int_{[0,1]^d}f(x){\rm d}x$ denote the multivariate integration problem defined over the sequence of Korobov spaces $H_{d,s,\gamma}$,
where $s>1/2$ and $\gamma=(\gamma_{d,j})_{d\in\N, 1\le j\le d}$ is a bounded sequence. Let $\widetilde H_{d,s,\gamma}$ be the $3^d$-dimensional
subspaces of trigonometric polynomials of degree at most one in each variable in $H_{d,s,\gamma}$.%
\begin{enumerate}
\item [(i)] If $(\mathbf S_d)$ is strongly polynomially tractable on $\widetilde H_{d,s,\gamma}$, then
\[
\displaystyle \sup_{d\in\N}\,\sum_{j=1}^d\gamma_{d,j} \,<\, \infty.
\]
\item [(ii)] If $(\mathbf S_d)$ is polynomially tractable on $\widetilde H_{d,s,\gamma}$, then 
\[
\displaystyle  \limsup_{d\to\infty}\, \frac{\sum_{j=1}^d\gamma_{d,j}}{\ln(d+1)} \,<\, \infty.
\]
\item [(iii)] If $(\mathbf S_d)$ is weakly tractable on $\widetilde H_{d,s,\gamma}$, then
\[
\displaystyle \lim_{d\to\infty}\,\frac{1}{d}\sum_{j=1}^d\gamma_{d,j}\,=\,0.
\]
\end{enumerate}
\end{prop}
\begin{proof}
We put $\alpha_{d,j}=\sqrt{2\gamma_{d,j}}\cdot (2\pi)^{-s}$ and obtain by Corollary~\ref{thm:alpha}
\[
n(\varepsilon,\mathbf S_d)\ge(1-\varepsilon^2)\prod_{j=1}^d(1+\alpha_{d,j}^2)= (1-\varepsilon^2)\prod_{j=1}^d \Bigl(1+2\gamma_{d,j}(2\pi)^{-2s}\Bigr).
\]
If $(\mathbf S_d)$ is strongly polynomially tractable, we observe from
\[
n(\varepsilon,\mathbf S_d)\ge (1-\varepsilon^2) \cdot 2\cdot (2\pi)^{-2s}\sum_{j=1}^d\gamma_{d,j}
\]
that $\sum_{j=1}^d\gamma_{d,j}$ must be uniformly bounded in $d\in\N$.
\medskip

If $(\mathbf S_d)$ is polynomially tractable or weakly tractable, we use the boundedness of $\gamma$ to estimate
\begin{equation*}%
\ln n(\varepsilon,\mathbf S_d)\ge \ln(1-\varepsilon^2)+\sum_{j=1}^d \ln\Bigl(1+2\gamma_{d,j}(2\pi)^{-2s}\Bigr)
\ge \ln(1-\varepsilon^2) + C \sum_{j=1}^d\gamma_{d,j}.
\end{equation*}
This estimate proves both (ii) and (iii).
\end{proof}

\section{New variants of Schur's Theorem}
\label{sec:new-schur}

In this section we present several variants of the uniform lower 
bound for the Schur
product obtained in \cite{Vy19}
and several consequences for the tractability of numerical integration. 

\subsection{Modifications of Schur's Theorem}

The first generalization of Proposition \ref{prop:main2} 
deals with matrices with reduced rank. 
This was first observed in \cite{Inder}.

\begin{thm}[\cite{Inder}]\label{thm:1} 
Let $M\in\R^{n\times n}$ be a positive semi-definite matrix with 
rank $r$.
Then
$$
M\circ M\succeq \frac{1}{r}({\rm diag} M)({\rm diag} M)^T.
$$
\end{thm}
\begin{proof}
If $M$ has also ones on the diagonal, then we may use the truncated singular value decomposition of $M$
and observe that $M$ is a Gram matrix of a set of $n$ vectors in ${\mathbb S}^{r-1}$.
The rest of the proof then follows in the same way as in Proposition \ref{prop:main2}.
Alternatively, one may follow \cite{Vy19} but write $M=AA^T$, where $A\in\R^{n\times r}$.
\end{proof}

The next version deals with the Schur product of two possibly 
different matrices $M\not=N.$ In this sense, it addresses a problem left open in \cite{Vy19}.

\begin{thm}\label{thm:2} Let $M,N\in\R^{n\times n}$ be positive semi-definite matrices
with $M=AA^T$ and $N=BB^T$ with $A,B\in \R^{n\times D}$ and $D\ge\max({\rm rank }(M),{\rm rank }(N))$.
Then, for every $c\in\R^n$,
\begin{equation}\label{eq:Schur:x}
\sum_{j,k=1}^n c_jc_k M_{j,k}N_{j,k}\ge \frac{1}{D}\Bigl(\sum_{j=1}^n c_j\langle A^j,B^j\rangle\Bigr)^2,
\end{equation}
where $A^j,B^j$ are the rows of $A$ and $B$, respectively.
\end{thm}

\begin{proof}
The proof is again similar to \cite{Vy19}. We write
\begin{align*}
\sum_{j,k=1}^n c_jc_k M_{j,k}N_{j,k}&=
\sum_{j,k=1}^n c_jc_k\sum_{l=1}^D A_{j,l}A_{k,l}\sum_{m=1}^D B_{j,m}B_{k,m}\\
&=\sum_{l,m=1}^D\Bigl(\sum_{j=1}^n c_jA_{j,l}B_{j,m}\Bigr)^2
\ge \sum_{l=1}^D \Bigl(\sum_{j=1}^n c_jA_{j,l}B_{j,l}\Bigr)^2\\
&\ge \frac{1}{D}\Bigl(\sum_{j=1}^n c_j \sum_{l=1}^D A_{j,l}B_{j,l}\Bigr)^2.
\end{align*}
\end{proof}

\begin{rem}
Using $(AB^T)_{j,j}=\langle A^j,B^j\rangle$, 
the estimate
\eqref{eq:Schur:x} can be written as
$$
M\circ N\succeq \frac{1}{D}({\rm diag }(AB^T))({\rm diag}(AB^T))^T.
$$
\end{rem}

The last generalization of 
Schur's Theorem, that, in a sense,
combines Theorem \ref{thm:1} and Theorem \ref{thm:2},
is the one we shall use later on.

\begin{thm}\label{thm:3} Let $M\in\R^{n\times n}$ be a positive semi-definite matrix
with rank $r$.
Let $M=AA^T=BB^T$ with $A,B\in\R^{n\times D}$
for some $D\ge r.$ 
Then, for every $c\in\R^n$,
\begin{equation}\label{eq:thm:3}
\sum_{j,k=1}^n c_jc_k M^2_{j,k}\ge \frac{1}{2r}\Bigl(\sum_{j=1}^n c_j\langle A^j,B^j\rangle\Bigr)^2,
\end{equation}
where $A^j,B^j\in\R^{D}$ are the rows of $A$ and $B$, respectively.
\end{thm}
\begin{proof}
We show that there exist two matrices $G,H\in\R^{n\times 2r}$ 
with rows denoted by $G^j$ and $H^j$, respectively,
such that $M=GG^T=HH^T$ and $\langle G^j,H^j\rangle=\langle A^j,B^j\rangle$ 
for every $j=1,\dots,n$. The proof then follows by an application
of Theorem \ref{thm:2} with $M=N$ and $2r$ instead of $r$.

Using the singular value decomposition theorem, we can write 
$A=U\Sigma V^T$ and $B=U\Sigma W^T$, where $U\in\R^{n\times r},\Sigma \in \R^{r\times r}$
and $V,W\in \R^{D\times r}$. Here, $U,V$ and $W$ have 
orthonormal columns and $\Sigma$ is a diagonal matrix. Furthermore, 
\begin{align*}
\langle A^j,B^j\rangle 
= (AB^T)_{j,j}=e_j^T(U\Sigma V^T)(W\Sigma U^T)e_j=\varepsilon_j^TV^TW\varepsilon_j,
\end{align*}
where $\varepsilon_j=\Sigma U^Te_j\in\R^r$ 
and $(e_j)_{j=1}^n$ is the canonical basis of $\R^n$. In the same way, 
we are looking for $G=U\Sigma X^T$ and $H=U\Sigma Z^T$ with matrices $X,Z\in\R^{2r\times r}$
with orthonormal columns and
\begin{equation}\label{eq:Had:1}
\langle G^j,H^j\rangle=\varepsilon_j^T X^TZ\varepsilon_j
=\varepsilon_j^TV^TW\varepsilon_j=\langle A^j,B^j\rangle,\quad j=1,\dots,n.
\end{equation}

The matrix $V^TW$ is formed by the scalar products of the column vectors of $V$ and $W$, respectively.
Using an orthogonal projection onto their common linear 
span (which has dimension at most $2r$), we can find $X,Z\in \R^{2r\times r}$ such that
$X^TZ=V^TW$, which is even stronger than \eqref{eq:Had:1}.
\end{proof}

\subsection{Applications to numerical integration}

Theorem \ref{thm:3} allows us to extend Theorem 
\ref{thm:curse_tensor_product} to a larger class of tensor product problems
with $e(0,\mathbf S_d) = \Vert \mathbf h_d \Vert =1$.

\begin{thm}\label{prop:unified}
Let $M$ be a reproducing kernel on a set $D$ and let $K=M^2$. Denote by $H(M)$ and $H(K)$ the Hilbert spaces with reproducing kernel $M$ and $K$, respectively.
Let $(b_\ell)_{\ell\in I}$ and $(\tilde b_\ell)_{\ell\in I}$ be two orthonormal bases of $H(M)$ and
\[
 g = \sum_{\ell\in I} b_\ell \tilde b_\ell \in H(K).
\]
We consider the normalized problem $S=\langle\cdot,h\rangle$ with $h=g / \Vert g\Vert$ on $H(K)$.
Then
\[
 e(n,S)^2 \ge 1 - \frac{2n}{\Vert g\Vert^2}.
\]
\end{thm}

\begin{proof}
For any $x,y\in D$, we have
\[
 M(x,y)=\sum_{\ell\in I} b_\ell(x) b_\ell(y) = 
 \sum_{\ell\in I} \tilde b_\ell(x) \tilde b_\ell(y). %
\]
Let $x_1,\hdots,x_n\in D$ and let $M=(M(x_j,x_k))_{j,k\le n}$. 
Then
\[
 M = B B^T = \tilde B \tilde B^T,
\]
where $B=(b_\ell(x_j))_{j\le n,\ell\in I}$ and $\tilde B=(\tilde b_\ell(x_j))_{j\le n,\ell\in I}$.
Theorem~\ref{thm:3} yields that
\[
 \sum_{j,k=1}^n c_jc_k M^2_{j,k}
 \ge \frac{1}{2n}\Bigl(\sum_{j=1}^n c_j(B \tilde B^T)_{jj}\Bigr)^2
 = \frac{\Vert g\Vert^2}{2n}\Bigl(\sum_{j=1}^n c_j  h(x_j) \Bigr)^2
\]
and thus the desired lower bound follows from Proposition~\ref{prop:error_vs_definiteness}.
\end{proof}

Let us observe that Theorem~\ref{thm:curse_tensor_product} is obtained
by considering particular orthonormal bases of $H(M_d)$. Namely, we take 
\begin{equation}\label{eq:bl}
 b_\ell(x)=%
\prod_{i=1}^d e_{\ell_i}(x_i)
 \quad\text{for}\quad 
 \ell\in \{1,2\}^d
\end{equation}
and $\tilde b_{\ell}=b_{\ell}$, $\ell\in\{1,2\}^d$.
Then we have
\[
 g(x)=\prod_{i=1}^d \left(e_1(x_i)^2 +e_2(x_i)^2\right).%
\]
and we obtain Theorem~\ref{thm:curse_tensor_product} (up to a factor 2). %

Another interesting choice of $(b_{\ell})_{{\ell}\in I}$ 
and $(\tilde b_{\ell})_{{\ell}\in I}$ of $H(M_d)$ is the following. We take again
$b_{\ell}$ defined by \eqref{eq:bl} and 
\begin{equation}\label{eq:bl'}
 \tilde b_\ell(x)=\prod_{i=1}^d \tilde e_{\ell_i}^{(i)}(x_i) 
 \quad\text{for}\quad 
 \ell\in \{1,2\}^d
\end{equation}
where
\[
 \left(\begin{array}{c}
 \tilde e_1^{(i)}\\
 \tilde e_2^{(i)}
 \end{array}\right)
 =
 U_i
 \left(\begin{array}{c}
 e_1\\
 e_2
 \end{array}\right)
\]
and $U_i\in\R^{2\times 2}$ is an orthogonal matrix. 

If $U_i$ is the identity matrix, we obtain $\tilde e^{(i)}_1=e_1$, $\tilde e^{(i)}_2=e_2$ and
$
\tilde e^{(i)}_1\cdot e_1+\tilde e^{(i)}_2\cdot e_2=e_1^2+e_2^2.
$
If, on the other hand, we choose
\[U_i=\left(\begin{matrix} \cos\varphi_i & \sin\varphi_i\\
\sin\varphi_i&-\cos\varphi_i
\end{matrix}\right),\quad \varphi_i\in[0,2\pi]
\] being a reflection across a line with angle $\varphi_i/2$, we obtain
\[
\tilde e^{(i)}_1\cdot e_1+\tilde e^{(i)}_2\cdot e_2=
\cos\varphi_i\cdot(e_1^2-e_2^2)+2\sin\varphi_i\cdot e_1e_2.
\]
Of course, we can mix these two examples by taking 
a different choice of $U_i$ for each dimension $i\le d$,
which leads to the following result.
\begin{cor}\label{cor:curse_general}
Let $F_1$ be a RKHS on $D_1$. 
Assume that there are functions $e_1$ and $e_2$ on $D_1$
such that $e_1^2, e_2^2$ and $\sqrt{2} e_1 e_2$ are orthonormal in $F_1$.
Let %
$$
\mathbf h_d(x)=\prod_{i=1}^d h_i(x_i),
$$
where $h_i\in{\rm span}\{e_1^2+e_2^2\}
\cup{\rm span}\{e_1^2-e_2^2,e_1e_2\}$ has unit norm $\Vert h_i\Vert=1$.
Then the tensor product problem $\mathbf S_d=\langle \cdot, \mathbf h_d\rangle$ satisfies
\[
e(n,\mathbf S_d)^2 \geq 1 - n \,  2^{-d+1} .
 \]
In particular, it suffers from the curse of dimensionality.  
\end{cor}

For the next result we take again 
the space of trigonometric polynomials of degree 1, 
see Corollary~\ref{cor1}.

\begin{cor}\label{cor:curse} 
Let $F_1$ be the RKHS on $[0,1]$ with the  orthonormal system 
$1$, $\cos (2\pi x)$ and $\sin (2\pi x)$.  
Let $d\ge 2$ and let $\{\varphi_i\}_{i=1}^\infty\subset[0,2\pi]$ be a bounded sequence.
Let
\[
\mathbf h_d(x)=\prod_{i=1}^d h_i(x_i),
\]
where
\begin{equation}\label{eq:phi1}
h_i(x_i)= \cos\varphi_i 
\cdot \cos(2\pi x_{i})+\sin\varphi_i\cdot \sin(2\pi x_{i})=\cos(2\pi x_i-\varphi_i)
\end{equation}
or $h_i=1$. Then the corresponding problem $\mathbf S_d=\langle \cdot,\mathbf h_d\rangle$ satisfies
\[
e(n,\mathbf S_d)^2 \geq 1 - n \,  2^{-d+1}
\]
and the problem suffers from the curse of dimensionality.
\end{cor}

\begin{rem}
Let us reformulate Corollary \ref{cor:curse} as an integration problem.
As in Section \ref{sec:trig}, we 
denote again $e_1(x)= 2^{1/4} \cos(\pi x)$ and $e_2=2^{1/4} \sin(\pi x)$ on $[0,1]$.
Let $\varphi\in[0,2\pi]$ and let $h(x)=\cos(2\pi x-\varphi), x\in[0,1]$, cf. \eqref{eq:phi1}.
Then 
$$
h(x)=\cos\varphi\cdot\frac{e_1^2(x)-e_2^2(x)}{\sqrt{2}}+\sin\varphi\cdot\sqrt{2}e_1(x)e_2(x).
$$
Consequently, if we define $S(f)=\langle f,h\rangle$ for $f\in F_1$, it satisfies
$$
S(e_1^2)=\frac{\cos\varphi}{\sqrt{2}},
\quad S(e_2^2)=-\frac{\cos\varphi}{\sqrt{2}},\quad S(\sqrt{2}e_1e_2)=\sin\varphi
$$
and we obtain
$$
S(f)=2\int_0^1 f(x)\cos(2\pi x-\varphi)\,{\rm d}x,\quad f\in F_1.
$$

Similarly, if we denote in Corollary \ref{cor:curse} 
by $I\subset\{1,\dots,d\}$ those indices, for which $h_i$ is given by \eqref{eq:phi1},
then
$$
\mathbf S_d(f)=\langle f,\mathbf h_d\rangle=\int_{[0,1]^d} f(x)\prod_{i\in I} [2\cos(2\pi x_i-\varphi_i)]\,{\rm d}x.
$$
\end{rem}

We finish this section by a couple of open problems.

\begin{OP} 
We conjecture that Theorem \ref{thm:3} holds 
true with $\frac{1}{r}$ instead of $\frac{1}{2r}$ in \eqref{eq:thm:3}, see also \cite[Theorem 1.9]{Inder}. This
would allow
to improve the error bound in Corollary \ref{cor:curse_general} and Corollary \ref{cor:curse} to
\[
e(n,\mathbf S_d)^2 \geq 1 - n \,  2^{-d}.
\]
\end{OP}
\begin{OP} Corollary \ref{cor:curse_general} 
shows the curse for all problems $\mathbf S_d=\langle \cdot, \mathbf h_d\rangle$, where
$\mathbf h_d=\bigotimes_{i=1}^d h_i$ is a tensor product
with all components $h_i$ being unit 
norm functions from either the span of $e_1^2+e_2^2$ or from the span of $e_1^2-e_2^2$
and $e_1e_2.$ Is the same true if we allow arbitrary $h_i\in F_1$?
\end{OP}

\section{Randomly chosen sample points} 
\label{sec:random}

We continue our analysis of high dimensional integration problems.
In the previous sections, we mainly studied the quality of optimal sample points.
Optimal sample points are usually hard to find. 
In this section, we switch our point of view and ask for the quality of random point sets 
$\X_n=\{x_1,\hdots,x_n\}$, where the points $x_1,\hdots,x_n$ are independent and identically distributed
in the domain of integration. 
With this, we continue the studies from \cite{HKNPU1,HKNPU2,KU19} on the quality of random information.
Like for optimal points, one can ask: How many random points do we need to
solve the integration problem up to the error $\varepsilon>0$?
Does this number depend exponentially on $d$, i.e., do we have the curse for random information?
We can use Proposition~\ref{prop:error_vs_definiteness} to obtain the following result
for Lebesgue integration on the unit cube
(where the interval $[0,1]$
may clearly be replaced by any other interval of length~1).

{\begin{thm}\label{thm:curse_rand}
Let $F_1$ be a RKHS on $[0,1]$ 
with (point-wise) non-negative 
and measurable kernel $K_1$ such that $1\in F_1$ 
is the representer of the integral.
We define
\[
 \kappa =   \esssup_{x\in [0,1]} K_1(x,x).
\]
Let $\X_n$ be a set of $n$ independent and uniformly distributed points in $[0,1]^d$.
If $\kappa>1$, 
then for any $\delta>0$ there are constants $c>0$ and $a>1$ such that
  \[
  \E\left[ e(\X_n,\mathbf S_d)^2 \right]
  \ge 1-\delta
 \]
 for all $n,d\in\N$ with $n\le c a^d$.
If $\kappa= 1$, then $e(\X_n,\mathbf S_d)=0$ holds almost surely for all $n,d\in\N$.
\end{thm}

This means that the tensor product problem for random information
is either trivial or suffers from the curse of dimensionality.
Note that the case $\kappa<1$ does not occur.
Moreover, if the kernel is continuous, the condition $\kappa>1$ is equivalent to the claim
that $F_1$ is at least two-dimensional.
This will become apparent in the proof.

\begin{proof}
 First note that the initial error is given by
 \[
  e(0,S_1)^2 = \Vert 1\Vert^2 = \langle 1,1\rangle = \int_0^1 1\, {\rm d}x = 1,
 \]
 so that the problem is properly normalized.
 Moreover, since $K_1(x,\cdot)$ can be written as the sum of 
 the orthogonal functions 1 and $K_1(x,\cdot)-1$,
 we have $K_1(x,x)=\Vert K_1(x,\cdot) \Vert^2 \ge \Vert 1\Vert^2=1$
 for all $x\in [0,1]$.
 Thus, $\kappa=1$ means that $K_1(x,x)=1$ almost everywhere.
 If $\kappa>1$, then there is some $c_1>1$ such that
 the set of all $x\in [0,1]$ with $K_1(x,x)\ge c_1$ has
 Lebesgue-measure $p>0$.
 We consider independent and uniformly distributed points 
 $x_1,x_2,\hdots\in[0,1]^d$.
 \medskip
 
 Let $\kappa=1$. Then we have almost surely that $\mathbf K_d(x_1,x_1)=1$ 
 and the one-dimensional matrix $\mathbf K_d(x_1,x_1)-\alpha$ is positive semi-definite if, and only if, $\alpha\le 1$. 
 Proposition~\ref{prop:error_vs_definiteness} yields that $e(x_1,\mathbf S_d)=0$
 as claimed.
 \medskip
 
 Let now $\kappa>1$. For all $i\in\N$, the number of coordinates of 
 $x_i$ with $K_1(x_{i,k},x_{i,k})\ge c_1$ is distributed according to the
 binomial distribution $B(d,p)$.
 Thus, the number of such coordinates is typically $pd$
 and greater than $pd/2$ with high probability. Namely,
 we have ${\mathbf K}_d(x_i,x_i)\ge c_1^{pd/2}$
 with probability at least $1-\exp(-p^2d/2)$.
 We put $c_2=c_1^{p/2}$ such that $c_1^{pd/2}=c_2^d$.
 For different indices $i,j\in\N$, Fubini's theorem yields that
 \[
  \E\,\mathbf  K_d(x_i,x_j) = \E\, \langle \mathbf  K_d(x_i,\cdot), 1 \rangle = 1.
 \]
 Using Markov's inequality and the non-negativity of ${\mathbf K}_d$, 
 this implies that
 $\mathbf K_d(x_i,x_j) \le c_2^d/(2n)$
 with probability at least $1-2n/c_2^d$.
 By a union bound, all these inequalities
 hold simultaneously for all $i,j\le n$
 with probability at least $1-n^3/c_2^d-n\exp(-p^2d/2)$.
 In this case, we have
 \begin{multline*}
  \left| \mathbf K_d(x_i,x_i) - \frac{c_2^d}{2n} 
  \right| - \sum_{j\neq i} \left| \mathbf K_d(x_i,x_j) - \frac{c_2^d}{2n} \right|\\
  \ge \mathbf K_d(x_i,x_i) - \frac{c_2^d}{2n}-\sum_{j\neq i}\Bigl(\mathbf K_d(x_i,x_j) + \frac{c_2^d}{2n}\Bigr)
  > c_2^d - \frac{c_2^d}{2} - \frac{c_2^d}{2} = 0
 \end{multline*}
 for all $i\le n$.
 Therefore, the matrix $(\mathbf K_d(x_i,x_j)-c_2^d/(2n))_{i,j\le n}$ is diagonally dominant
 and hence positive definite by Gershgorin circle theorem \cite[Theorem 6.1.1]{HoJo13}. Proposition~\ref{prop:error_vs_definiteness} implies that
 \[
 e(\X_n,\mathbf S_d)^2\ge 1 - \frac{2n}{c_2^d}
 \]
 with the stated probability.
 This yields the assertion.
\end{proof}}

 Let us now consider two examples.
 The first example shows that random information 
 may be much worse than optimal information.
 We consider the space $F_1$ of affine linear functions on $[0,1]$ 
 with scalar product
 \[
 \langle f, g \rangle 
 = \langle f, g\rangle_2 +
 \langle f^\prime, g^\prime\rangle_2.
 \]
 The kernel of this space is given by
 \[
  K_1(x,y)=1+\frac{12}{13}\left(x-\frac12\right)\left(y-\frac12\right).
 \]
 The tensor product space $\mathbf F_d$ consists of all
 functions $f\colon[0,1]^d\to\R$ which are affine linear in each variable.
 This space satisfies the conditions of Theorem~\ref{thm:curse_rand}.
 On the other hand, the integral of any such function is given by
 its function value at the center of the cube. 
 This means that the integration problem on $\mathbf F_d$ 
 is trivial for optimal information, but suffers from the curse
 of dimensionality if we only have random information.
 \medskip

 As another example,
 let us consider the integration problem on 
 the space $F_1$ of polynomials with degree at most 2 %
 with scalar product
 \[
 \langle f, g \rangle 
 = \langle f, g\rangle_2 + 
 \langle f^\prime, g^\prime\rangle_2 + \langle f^{\prime\prime}, g^{\prime\prime}\rangle_2.
 \]
The tensor product space $\mathbf F_d$ consists of polynomials 
of degree 2 or less in every variable.
 It was raised as an open problem in \cite[Open Problem~44]{NW10} whether
 the integration problem on the tensor product space $\mathbf F_d$
 suffers from the curse of dimensionality.
 For optimal point sets, this question remains unsolved.
 For random point sets, Theorem~\ref{thm:curse_rand} yields the curse of dimensionality. 
 The assumptions are readily verified
 with the kernel $K_1$ given in \cite{NW10}. In fact,
 the reproducing kernel %
 even satisfies the condition
 \begin{equation}\label{cond:large-diagonal}
  \inf_{x\in[0,1]} K_1(x,x) > 1.
 \end{equation}
One may ask 
whether this condition is already enough
to obtain the curse of dimensionality for optimal information.

\begin{OP}
 Let $F_1$ be a RKHS of functions on $[0,1]$
 with non-negative kernel $K_1$ satisfying~\eqref{cond:large-diagonal}
 such that $1\in F_1$ is the representer of the integral.
 Prove or disprove that Lebesgue integration on 
 the tensor product space $\mathbf F_d$ 
 with optimal information suffers from the curse of dimensionality.
\end{OP}

\medskip 
\noindent
{\bf Acknowledgements: } \
We thank Dmitriy Bilyk, Mario Ullrich, Henryk Wo\'zniakowski 
and the referees for valuable comments.
We also thank the Oberwolfach team and the organizers of the workshop
``New Perspectives and Computational Challenges in High Dimensions'' (ID 2006b, February 2020);
much of this work was done and discussed during this workshop.
A.\,Hinrichs and D.\,Krieg gratefully acknowledge the support 
by the Austrian Science Fund (FWF) Project F5513-N26, 
which is a part of the Special Research Program ``Quasi-Monte Carlo Methods: Theory and Applications''.
The research of J.~Vyb\'\i ral was supported by the 
grant P201/18/00580S of the Grant Agency of the Czech Republic
and by the European Regional Development Fund-Project ``Center 
for Advanced Applied Science'' \\
(No. CZ.02.1.01/0.0/0.0/16\_019/0000778).

\end{document}